\documentclass[leqno]{amsart} 
\usepackage[utf8]{inputenc}
\usepackage[foot]{amsaddr}
\usepackage{latexsym}
\usepackage{amssymb}
\usepackage{amsmath}
\usepackage{amsthm}
\usepackage{amsfonts}

\usepackage{soul, color}
\usepackage{dsfont}
\usepackage{verbatim}
\usepackage{graphicx}
\usepackage{graphics}
\usepackage{enumerate}
\usepackage{url}
\usepackage{mathtools}
\usepackage{comment}
\usepackage{bbm}
\usepackage{hyperref}
\usepackage{relsize}
\usepackage{subcaption}
\usepackage{tikz}
\usetikzlibrary{shapes,arrows,positioning,calc}
\allowdisplaybreaks

\usetikzlibrary{decorations.pathmorphing}

\def\n{{\hat{\mathbf{n}}}}
\def\r{{\mathbf{r}}}

\def\bv{{\hat{\mathbf{v}}}}
\def\x{{\mathbf{x}}}
\def\y{{\mathbf{y}}}

\def\Omegastar{\Omega_{\textrm{star}}}
\def\Omegarandom{\Omega_{\textrm{random}}}
\def\Omegacavity{\Omega_{\textrm{cavity}}}
\def\Omegae{\Omega_{\textrm{ellipsoid}}}
\def\Omegat{\Omega_{\textrm{torus}}}
\def\gammastar{\gamma_{\textrm{star}}}
\def\gammarandom{\gamma_{\textrm{random}}}
\def\gammacavity{\gamma_{\textrm{cavity}}}

\def\gammat{\gamma_{\textrm{torus}}}
\def\uex{u_{\textrm{exact}}}

\newtheorem{theorem}{Theorem}[section]
\newtheorem{corollary}{Corollary}[theorem]
\newtheorem{lemma}[theorem]{Lemma}
\theoremstyle{definition}

\theoremstyle{remark}
\newtheorem{remark}{Remark}

\begin{document}
\title{String kernel representations in elastostatics}
\author{Jeremy G. Hoskins$^1$}
\address{$^1$Department of Statistics, University of Chicago, USA and NSF-Simons National Institute for Theory and Mathematics in
Biology, Chicago, IL.}
\email{jeremyhoskins@uchicago.edu}

\author{Alan E. Lindsay$^2$}\address{$^2$Department of Applied and Computational Mathematics and Statistics, University of Notre Dame, Notre Dame, IN, 46556, USA.}
\email{a.lindsay@nd.edu}

\author{Manas Rachh$^3$}
\address{$^3$Department of Mathematics,
Indian Institute of Technology Bombay,
Powai, Mumbai - 400076.}
\email{mrachh@iitb.ac.in}
\thanks{Corresponding author: mrachh@iitb.ac.in}


\begin{abstract}
In this paper we present a new boundary integral equation formulation for the solution of the elastostatic traction boundary value problem in two and three dimensions. The approach relies on the introduction of new layer potentials, called \emph{string kernels}, which are based on modifications of the Boussinesq-Cerruti family of half-space solutions. We prove that the resulting integral equations are second-kind integral equations, and show that they are well-behaved in the incompressible limit. We illustrate the performance of the method with several numerical examples.
\end{abstract}
\maketitle

\section{Introduction}

In this paper we consider the problem of linear elastostatics with traction boundary conditions in two and three dimensions, modeled by the boundary value problem
\begin{align}
    L[u](\x):=(\lambda+\mu) \nabla (\nabla\cdot u(\x)) + \mu \Delta u(\x) = 0, \quad &\x \in \Omega,\label{eqn:elasto}\\
     \n(\x)\cdot \sigma[u](\x) = f(\x), \quad &\x \in \partial \Omega,\label{eqn:bc}
\end{align}
where $\lambda$ and $\mu$ are the Lam\'{e} parameters, $u$ is an unknown displacement field, $\n(\x)$ denotes the outward normal, and the stress $\sigma$ is defined by
$$\sigma_{ij}[u] = \lambda \delta_{ij} \frac{\partial u_k}{\partial x_k} + \mu \left(\frac{\partial u_j}{\partial x_i}+\frac{\partial u_i}{\partial x_j} \right),$$
where $\delta_{ij}$ is the Kronecker delta function and we employ the standard Einstein summation convention. {Throughout, we let
$\alpha = \frac{\lambda + \mu}{\lambda + 2\mu}$
denote the non-dimensional constant characterizing the relative strength of the
Lam\'{e} parameters; in terms of the Poisson ratio $\nu$ it is given by
$\alpha = 1/(2(1-\nu))$, so that $\alpha \in (1/4,1)$ for all physically
admissible $\nu \in (-1,1/2)$, with $\alpha \to 1^{-}$ in the incompressible
limit.} In what follows we will assume that $\partial \Omega$ is infinitely differentiable and compact, though $\Omega$ can be either bounded or unbounded. It is well-known (see \cite{hsiao2008boundary} for example) that for the interior problem the above boundary value problem has a finite-dimensional nullspace corresponding to rigid motions. Existence requires the applied traction to have zero net force and zero net torque. On unbounded domains, one must also impose suitable decay conditions at infinity.

There is a large literature on solving the traction boundary value problem using finite difference and finite element methods, see~\cite{bochev2011energy,zienkiewicz2005finite,arnold1984family,dupont1979family,brenner1992linear,falk1991nonconforming,falk2008finite}, for example. Direct discretization of the PDE, however, carries several well-known difficulties, particularly for complex domains requiring adaptivity, since the condition number of the resulting sparse linear systems grows as $h^{-2}$ under mesh refinement. Additionally, exterior and unbounded problems require artificial domain truncation. These issues are compounded as the material approaches the incompressible limit $\lambda \to \infty$: standard conforming displacement-based finite element methods can exhibit volumetric locking, and the resulting linear systems become increasingly ill-conditioned, with error and conditioning bounds that are not uniform in $\lambda$ \cite{babuskaLockingEffectsFinite1992,brenner2007mathematical,boffi2013mixed}. Recovering uniform behavior requires non-conforming, mixed, or stabilized formulations that satisfy an inf–sup (LBB) condition, as in the incompressible-Stokes case to which elastostatics formally reduces in this limit.

Boundary integral equation (BIE) methods, by contrast, discretize only the boundary of the domain. This leads to a reduction in the dimension of the problem, straightforward treatment of exterior problems and unbounded domains, and the treatment of complex geometries without volumetric meshing. Moreover, for many PDEs and for suitably chosen ansatzes, this reduction to a BIE produces second-kind integral equations whose condition number remains bounded independent of discretization size \cite{kress_2014}. A further practical advantage of the BIE framework is that derived quantities such as stress are computed by applying the relevant differential operator to the kernel, and retain the full convergence rate of the boundary discretization at any point bounded away from $\partial \Omega.$ With appropriate near-singular quadrature or adaptive integration, these quantities can frequently be accurately computed up to and on $\partial \Omega$ itself.

In two dimensions, the standard integral-equation-based approach for solving the traction boundary value problem is to use the Sherman-Lauricella representation~\cite{lauricella,sherman}. This approach has subsequently been extended to traction boundary value problems on multiply connected domains in~\cite{greengard1996integral,helsing2000interior}. For a detailed discussion and related problems in elasticity, see~\cite{parton1982integral,mikhlin2014integral,muskhelishvili1953some,jaswon1978integral,7a5ceb47-1563-3f2e-8d35-1a46dd4eba65}. However, the Sherman-Lauricella representation does not extend naturally to three dimensions, as the method relies on enforcing indefinite integrals of the surface traction, rather than the value of the surface traction itself.

 Another class of integral-equation-based approaches involves \emph{direct representations}, which are based on the Somigliana identity (the elastostatic analog of Green's identity). This approach tends to result in integral equations involving singular or hypersingular integral operators in two and three dimensions~\cite{bonnet,hsiao2008boundary,vogel1973integral,chen1999review2}. Alternatively, indirect representations seek to represent the solution to the PDE as integrals of unknown densities on the boundary $\partial \Omega$ against kernels constructed from linear combinations of the free-space elastostatics Green's function and its derivatives. This choice also results in a singular integral operator upon imposing the boundary conditions~\cite{vogel1973integral,hong1988derivations}.
 
 In both the typical indirect and direct representations, the resulting integral operator is Fredholm of index zero but is not a compact perturbation of the identity, in contrast to the Laplace Neumann problem, where the single-layer ansatz produces the classical second-kind equation $\frac{1}{2}I + K^*.$ Consequently, the standard machinery for second-kind equations is not directly applicable. Nevertheless, a variety of quadrature methods have been developed to address singular and hypersingular integral operators, see~\cite{gimbutas2012calculation,andra1998integration,schwab1992numerical,karaiev2020singular,yao2024robust}, for example.

Our approach represents the displacement field using a layer potential on $\partial \Omega$ whose kernel is built from modifications of the Boussinesq-Cerruti family of half-space fundamental solutions~\cite{boussinesq1885application,cerruti1882ricerche,love2013treatise}, i.e., the Green's functions for a point load acting in a half-space with zero traction boundary conditions. Since the Boussinesq–Cerruti solutions already satisfy the zero-traction condition on a flat half-space boundary, they are an appealing candidate kernel for traction boundary value problems on general domains. They do, however, carry a non-local singularity: in addition to the expected singularity at the source, each kernel is singular along a half-line emanating from the source in the direction complementary to the half-space. This geometric obstruction rules out their direct use as layer-potential kernels on non-convex domains and for exterior problems, where such half-lines generically intersect the solution domain.

A key observation is that each Boussinesq–Cerruti kernel admits a representation as a line integral of certain derivatives of the free-space elastostatic fundamental solution along this half-line. Truncating the integration to a finite segment produces what we call a \emph{string kernel}: near the source the singularity structure inherited from the free-space kernel is preserved, while the extended singularity is now confined to a bounded segment which, by appropriate choice of string orientation and length, can be made to lie entirely outside $\overline{\Omega}$
 regardless of the domain geometry. The full Boussinesq–Cerruti representation is recovered as the string length tends to infinity. Using a string kernel as the integrand of a layer potential with a density supported on $\partial \Omega$ gives a second-kind boundary integral equation that is well-behaved in the limit $\lambda \to \infty,$ $\mu$ fixed.

A distinct, though conceptually related, class of methods is the method of fundamental solutions (MFS)~\cite{fairweather1998method,cao1991three,christiansen1976kupradze,keshavarzi1978modified,kupradze1967potential}, in which finite collections of point sources, called \emph{image charges}, are placed outside of the domain. The strengths of the image charges are inferred from symmetry, precomputed by a local auxiliary solve, or included as additional unknowns in the global linear system. As such, though effective in practice, the result does not typically correspond to a discretization of a second-kind integral equation, and hence comes with fewer analytic guarantees on conditioning and convergence. {For several combinations of geometry, PDE, and boundary conditions, exact image systems are available in closed form which involve continuous line sources. For the exterior of a sphere, the Neumann Green's function of the Laplace operator consists of two point images together with a uniform line image on the segment joining the center to an appropriate image point~\cite{neumann1883,dassios2012}; the image system for a sphere in Stokes flow likewise contains continuous distributions of Stokeslets, rotlets, stresslets and potential dipoles along that segment~\cite{microhydrodynamics}. Numerically, \cite{BROMS2025113636} exploits this structure, supplementing MFS sources on interior proxy surfaces with singularities discretized along line segments joining particle centers, in order to resolve interactions between nearly touching spheres. For arbitrary smooth geometries string kernel representations use a continuum of image charges confined to line segments, with relative strengths along each string are known a priori, requiring neither an auxiliary solve nor additional unknowns.}

The contributions of this paper are as follows. We introduce a family of \emph{string kernel} layer-potential representations for the traction boundary value problem of linear elastostatics in two and three dimensions, derived by truncating the line-integral representation of the Boussinesq–Cerruti half-space fundamental solutions. We prove that, for smooth $\partial \Omega$ and an appropriate choice of string orientation and length, the resulting boundary integral operator is a compact perturbation of a multiple of the identity, and hence that the associated boundary integral equation is of second kind. Naturally, the condition number of the representations deteriorates when the string length goes to zero. In two dimensions we show that the strings can be \emph{bent}, allowing for longer strings and almost entirely avoiding this issue. Furthermore, we show that the condition number is independent of $\lambda.$ Indeed, the integral equation itself is independent of $\lambda.$ In three dimensions, we present numerical experiments showing the behavior of the  conditioning of our representation as the string length goes to zero. Additionally, under certain assumptions, we show that the condition number of our integral equation is bounded as $\lambda\to \infty$ with $\mu$ held fixed, and in this limit reduces to a compact perturbation of a standard second-kind integral equation for the Stokes traction problem. 

The remainder of the paper is organized as follows. In Section \ref{sec:bous} we review the Boussinesq solutions and the properties of layer potentials generated using them. Next, in Section \ref{sec:string} we introduce string kernel representations and use them to derive our second-kind integral equation for elastostatics. In Section \ref{sec:num_app} we describe numerical considerations required when using string kernels, including stable evaluation and choice of string length. Following this, in Section \ref{sec:num_ill}, we illustrate the application of our approach to 2D and 3D examples.

\section{Boussinesq-Cerruti solutions in two and three dimensions}\label{sec:bous}

\subsection{Notation, definitions, and basic properties}
Our boundary integral representations for solving (\ref{eqn:elasto}-\ref{eqn:bc}) are based on suitable modifications of the fundamental solution of half-space elastostatics problems. In particular, we take as our starting point the Boussinesq-Cerruti solutions \cite{boussinesq1885application} in two and three dimensions for point loads on a semi-infinite isotropic elastic medium. We refer the reader to \cite{landau1959theory}, for example, for further discussion and derivations.

In two dimensions, given a source $\y = (y_1,y_2)^T,$ a target $\x = (x_1,x_2)^T,$ and a unit direction vector $\bv=(v_1,v_2)^T,$ we define the matrix-valued function $G^{B,2}$ by
\begin{align}
    G^{B,2}(\x,\y;\bv) &= 2G^{S,2}(\x,\y) +\frac{1}{2\pi \mu}\frac{1-\alpha}{\alpha}\left[{-\log(r){\rm I}_2}+ {\rm arg}(\bv^\perp\cdot \r,-\bv\cdot \r) \begin{pmatrix} \phantom{-}0 & 1 \\ -1 & 0 \end{pmatrix}\right]\nonumber
\end{align}
where $\r = \x - \y,$ $r = |\r|,$ $\bv^\perp = (v_2,-v_1),$ ${\rm arg}(u,v)$ denotes the argument of $v+iu$ with branch cut along the negative real axis, and $G^{S,2}$ is the 2D Stokeslet~\cite{microhydrodynamics,pozrikidis1992boundary}
\begin{align}
    G^{S,2}(\x,\y) = \frac{1}{{4}\pi \mu}\left[ -\log(r) {\rm I}_2 + \frac{\r \r^T}{r^2}\right].
\end{align}
Here ${\rm I}_d$ denotes the $d \times d$ identity matrix.

A straightforward calculation shows that 
$$L[G^{B,2}(\cdot,\y;\bv)](\x) = 0,$$
for all $\x \in \mathbb{R}^{2} \setminus \mathcal{L}_{\y,\bv}$, where $\mathcal{L}_{\y,\bv} = \{ \x \, | \, \x = \y + t \bv, \, t\in\mathbb{R}^{+} \}$, i.e., for all $\x$ not on the half-line starting at $\y$ and going to infinity in the $\bv$ direction. Moreover, again excluding this line, 
\begin{align}\label{eqn:sigg2}
\sigma[G^{B,2}(\cdot,\y;\bv)]_{ijk}(\x) = -2\frac{r_ir_jr_k}{\pi r^4}=2\, {T^{S,2}_{ijk}(\x,\y),}
\end{align}
{where $T^{S,2}$ denotes the traction kernel of the Stokes single-layer potential.}

Similarly, in three dimensions given a source $\y = (y_1,y_2,y_3)^T,$ a target $\x = (x_1,x_2,x_3)^T,$ and a unit direction vector $\bv=(v_1,v_2,v_3)^T,$ we define the matrix-valued function $G^{B,3}$ by
\begin{align}
    G^{B,3}(\x,\y;\bv) &= 2 G^{S,3}(\x,\y)  \nonumber\\
    &\quad + \frac{1}{4 \pi \mu} \frac{1-\alpha}{\alpha} \left[\frac{1}{R}{\rm I}_3 + \frac{\r \bv^T - \bv \r^T- (\r\cdot \bv) \bv \bv^T}{rR} - \frac{Q_\bv\r\r^TQ_\bv}{rR^2}\right],\nonumber
\end{align}
where $\r = \x - \y,$ $r = |\r|,$ $Q_\bv = {\rm I}_3 - \bv\bv^T,$ $R = r-\r \cdot \bv,$ where $G^{S,3}$ denotes the 3D Stokeslet~\cite{graham2018microhydrodynamics,BROMS2025113636,TORNBERG20081613}
\begin{align}
    G^{S,3}(\x,\y) = \frac{1}{{8} \pi \mu} \left[ \frac{{\rm I}_3}{r} + \frac{\r \r^T}{r^3} \right].
\end{align}

Again, a straightforward but tedious calculation shows that 
$$L[G^{B,3}(\cdot,\y;\bv)](\x) = 0,$$
at all points $\x$ not on the half-line $\mathcal{L}_{\y,\bv}$, and 
\begin{align*}
\sigma[G^{B,3}(\cdot,\y;\bv)]_{ijk}(\x) &= {2T^{S,3}_{ijk}(\cdot,{\bf y})}+\frac{1}{4\pi }\frac{1-\alpha}{\alpha}\left\{\left(\frac{4}{R} + \frac{2}{r} \right) \frac{(Q_\bv \r)_i (Q_\bv \r)_j \r_k}{r^2R^2}\right.\\
&\quad -2 \frac{(Q_{\bv})_{i,j} \r_k}{r} \left(\frac{1}{R^2}-\frac{1}{r^2}\right) -4\frac{(Q_\bv \r)_i(Q_\bv \r)_j \bv_k}{rR^3}\\
&\left. \quad + \frac{2(Q_\bv)_{i,j} \bv_k}{R^2} -2 \frac{(Q_\bv)_{i,k}(Q_\bv \r)_j+(Q_\bv)_{j,k}(Q_\bv \r)_i}{r R^2}\right\},
\end{align*}
{where $T^{S,3}_{ijk}(\x,\y) = -\frac{3 r_ir_jr_k}{4\pi r^5}$ is the traction kernel of the Stokes single-layer potential in 3D.}
\subsection{Boussinesq-Cerruti layer potentials}

In this section we review results on limits of {the traction kernels for Stokes single-layer potentials,} and prove analogous results for the {normal stress of potentials involving $G^{B,2}$ and $G^{B,3}.$} The following standard theorem characterizes the normal stress of Stokeslet layer potentials. 

\begin{theorem}\label{thm:gs23_lim}
Let $\Omega$ be a domain in $\mathbb{R}^d,$ $d=2,3$ with smooth boundary $\partial \Omega.$ For $\rho \in L^2(\partial \Omega;\mathbb{R}^d),$
{\begin{align*}\lim_{\delta\to 0^+} \n_i(\x) \,&\left[\int T^{S,d}_{ijk}(\cdot,\y)\rho_k(\y)\,{\rm d}S(\y)\right](\x-\delta \n)\\
&=\frac{1}{2}\rho_j(\x) + \int \n_i(\x)\,T^{S,d}_{ijk}({\bf x},\y) \rho_k(\y)\,{\rm d}S(\y),
\end{align*}}
in an $L^2$ sense. Moreover, the integral operator on the right-hand side is compact from $L^2(\partial \Omega;\mathbb{R}^d)$ to itself. 
\end{theorem}
\begin{proof}
The proof follows from standard arguments. For continuous densities, $\rho,$ the corresponding limits can be found in~\cite{hsiao2008boundary}, for example. From these, the $L^2$ statement can be obtained using a similar argument as in~\cite{kersten1980grenz,kress_2014}.
\end{proof}

We now turn to the {remainder kernels $$W^d_{ijk}(\cdot,\y; \hat{\bf v}):=\sigma[G^{B,d}(\cdot,\y;\hat{\bf v})]_{ijk}-2T^{S,d}_{ijk},$$ with the specific choice $\bv(\y) = \n(\y),$} the outward facing normal. {In three dimensions $W^3$ is singular on the half-line $\mathcal{L}_{\y,\bv},$ and so, where this half-line intersects the boundary, the stress will not be well-defined. In two dimensions, no such difficulty arises as $\sigma[G^{B,2}]$ has a removable singularity along the ray, and $W^{2} \equiv 0.$} This issue is a \emph{nonlocal} phenomenon in so far as the singularities arising near a point $\x \in \partial \Omega$ are due to the density at points $\y \in \partial \Omega$ which lie outside of an open neighborhood of $\x.$ In light of this, we analyze the kernels after multiplying by sufficiently narrow bump functions.

To that end, in the following we let $\chi: \mathbb{R} \to \mathbb{R}$ be a smooth function which is identically $1$ for $|x| \leq 1/2,$ and identically zero for $|x|\geq 1.$ For any $r>0$ set $\chi_r(x)=\chi(x/r).$ For a domain $\Omega$ in $\mathbb{R}^2$ or $\mathbb{R}^3$ with smooth boundary $\partial \Omega$ set 
$$\beta(\Omega) = \sup \{t \in \mathbb{R}^+\,|\, \x + s \n(\x) \notin \Omega, \,{\rm for}\,\,{\rm all}\,\,s \in (0,t),\,\,\x \in \partial \Omega\},$$
and $\beta'(\Omega)= {\rm min}\{1,\beta(\Omega)/2\}.$ For ease of exposition when there is no risk of confusion we will write $\beta'$ in place of $\beta'(\Omega)$ {{(see Figure~\ref{fig:R-geometry}).}}

\begin{theorem}\label{thm:g2_lim}
Let $\Omega$ be a domain in $\mathbb{R}^d,$ $d=2,3$ with smooth boundary $\partial \Omega,$ and $\beta'$ be defined as above. For $\rho \in L^2(\partial \Omega; \mathbb{R}^d),${\begin{align*}\lim_{\delta\to 0^+} \n_i(\x) &\int W^{d}_{ijk}({\bf x}-\delta \n(\x),{\bf y};\n(\y))\chi_{{\beta}'}(|\x-\delta\n(\x)-\y|)\rho_k(\y)\,{\rm d}S(\y)\\
&= \int \n_i(\x)W^{d}_{ijk}(\x,\y;\n(\y))\,\,\chi_{{\beta}'}(|\x-\y|) \rho_k(\y)\,{\rm d}S(\y),
\end{align*}}
in an $L^2$ sense. Moreover, the integral operator on the right-hand side is compact from $L^2(\partial \Omega; \mathbb{R}^d)$ to itself.
\end{theorem}
\begin{proof}
{Since $W^2 \equiv 0,$ all that remains is to prove the result in three dimensions.} 

\begin{figure}[ht]
\centering
\begin{tikzpicture}[
  line cap=round, line join=round, scale=2.6,
  every node/.style={font=\small}
]
  \fill[blue!5]
    plot[smooth, variable=\t, domain=-1.20:1.40, samples=100] ({\t},{-0.06*\t*\t})
    -- (1.40,-0.85) -- (-1.20,-0.85) -- cycle;

  \fill[blue!5]
    plot[smooth, variable=\t, domain=-1.20:1.40, samples=100] ({\t},{1.20+0.05*\t*\t})
    -- (1.40, 1.80) -- (-1.20, 1.80) -- cycle;

  \draw[thick]
    plot[smooth, variable=\t, domain=-1.20:1.40, samples=100] ({\t},{-0.06*\t*\t});

  \draw[thick]
    plot[smooth, variable=\t, domain=-1.20:1.40, samples=100] ({\t},{1.20+0.05*\t*\t});

  \node at (-0.95,-0.45) {$\Omega$};
  \node at (-0.95, 1.55) {$\Omega$};
  \node at (-0.95, 0.60) {$\Omega^c$};
  \node[anchor=west] at (1.40,-0.12) {$\partial\Omega$};
  \node[anchor=west] at (1.40, 1.22) {$\partial\Omega$};

  \coordinate (y)  at ( 0.15,-0.00135);
  \coordinate (x)  at ( 0.55,-0.0182);

  \coordinate (yn) at ( 0.160, 0.549);
  \coordinate (xn) at ( 0.586, 0.531);

  \coordinate (xd) at ( 0.522,-0.437);

  \draw[thin, dashed, black!75] (y) circle (0.50);
  \coordinate (bedge) at ($(y)+({0.50*cos(135)},{0.50*sin(135)})$);
  \draw[->, thin, black] (y) -- (bedge);
  \node[black] at ($(y)!0.5!(bedge)+(0.06,0.05)$) {$\beta'$};

  \draw[decorate, decoration={snake, amplitude=0.6mm, segment length=2.5mm},
        thick, purple!70!black]
    (y) -- ($(y) + ({0.012*1.18},{1.000*1.18})$);
  \node[purple!70!black, font=\footnotesize, anchor=west]
    at ($(y) + ({0.012*1.05+0.045},{1.000*1.05})$) {$\mathcal{L}_{\y,\n(\y)}$};

  \draw[->, thick, blue!70!black] (y) -- (yn)
    node[anchor=south east, inner sep=2pt] {$\n(\y)$};
  \draw[->, thick, blue!70!black] (x) -- (xn)
    node[anchor=south, inner sep=2pt] {$\n(\x)$};

  \draw[->, thick, red!70!black, dashed] (x) -- (xd);

  \draw[->, very thick, orange!80!black] (y) -- (xd);
  \node[orange!80!black]
    at ($(y)!0.50!(xd) + (0.10,0.04)$) {$\r'$};

  \fill[red!85!black] (y)  circle (0.030);
  \fill[red!85!black] (x)  circle (0.030);
  \fill[black!75]    (xd) circle (0.022);

  \node[anchor=east]       at ($(y) +(-0.030,0.045)$) {$\y$};
  \node[anchor=north west] at ($(x) +(0.030,-0.025)$) {$\x$};
  \node[anchor=west]       at ($(xd)+(0.040,0)$)     {$\x-\delta\,\n(\x)$};

\end{tikzpicture}
\caption{Geometry of sources, targets, and singularities.}
\label{fig:R-geometry}
\end{figure}
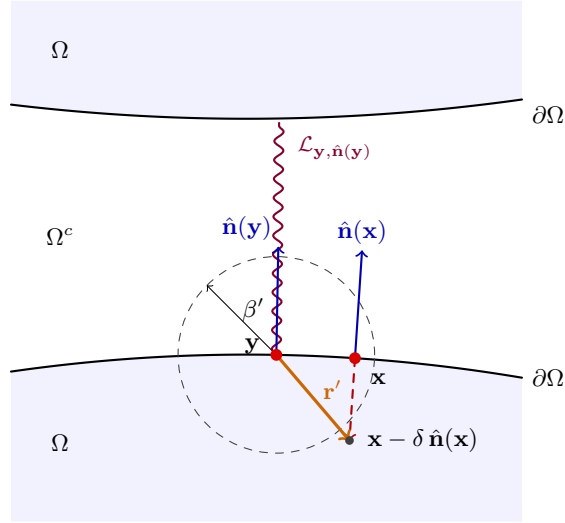
In 3D, we first use the smoothness of kernels away from $\partial \Omega$ to show that the operator with integral kernel {$$A_{jk}(\x,\y;\delta):=\n_i(\x)W^3_{ijk}(\x-\delta \n(\x),\y,\n(\y)) {\chi_{\beta'}}(|\x-\delta\n(\x)-\y|),$$} is compact for $\delta>0$ and sufficiently small, and uniformly bounded. We then show that the operator with $\delta=0$ is the norm limit of these operators and hence is also compact. Though the proof involves fairly standard arguments, we include it here for completeness. In the following we use $C$ to represent arbitrary constants which depend only on $\partial \Omega.$

For $\delta>0$ and sufficiently small, ${A_{jk}}(\x,\y;\delta)$ is smooth and hence the resulting integral operator is compact from $L^2(\partial \Omega; \mathbb{R}^3)$ to itself. Next, if $R(\x,\y;\delta):=|\x-\delta \n(\x)-\y| - (\x-\delta \n(\x) - \y) \cdot \n(\y)$ for $\x,\y \in \partial \Omega,$ then there exist constants $r_0,\delta_0,c>0,$ depending only on $\Omega,$ such that $|R(\x,\y;\delta)|>c |\x-\delta \n(\x)-\y|,$ whenever $|\x-\y|<r_0$ and $\delta<\delta_{0}$.  Moreover, without loss of generality we can suppose that $\delta_0$ is chosen sufficiently small so that ${A}(\x,\y;\delta)$ is smooth in $\x,\y \in \partial \Omega$ for all $0< \delta<\delta_0.$ With this bound on $R$ it follows that for all $|\x-\y|<r_0$ and $\delta<\delta_0,$
$$|{A_{jk}}(\x,\y;\delta)| \le C \frac{|Q_{\n(\y)} \n(\x)|}{r^2},\quad j,k=1,2,3,$$
for some constant $C,$ and with $r = |\x-\delta\n(\x)-\y|.$ If $\partial \Omega$ is twice continuously differentiable then Taylor's theorem (see~\cite{kress_2014} for example) implies that there exist constants $R_0$ and $C$ such that $|Q_{\n(\y)}\n(\x)|<C |\x-\y|$ whenever $|\x-\y| < R_0.$ Hence, for $\x$ and $\y$ sufficiently close, and $\delta$ sufficiently small,
$$|{A_{jk}}(\x,\y;\delta)| \le C\frac{1}{|\x-\y|}.$$

{Hence, the integral operators with kernels ${A}(\x,\y;\delta),$} are compact, and uniformly bounded from $L^2(\partial \Omega; \mathbb{R}^3) \to L^2(\partial \Omega; \mathbb{R}^3)$ for all non-negative $\delta$ less than some positive $\hat{\delta},$ depending only on $\partial \Omega.$ This follows from the fact that ${A}(\x,\y;\delta)$ has a weakly singular kernel bounded by $C/|\x-\y|$ with a constant $C$ which is independent of $\delta,$ see ~\cite{punchin1988weakly} for example.

Similar arguments give that for $0 \le \delta_1, \delta_2 <\delta_0,$
$$|{A}(\x,\y;\delta_1)-{A}(\x,\y;\delta_2)| \le C \frac{|\delta_1-\delta_2|}{r^2}.$$
For any $0< \delta <\delta_0$ and any $\rho \in L^2(\partial \Omega; \mathbb{R}^3),$ and any $\epsilon >0,$ we note that
\begin{align*}
\|\int_{\partial \Omega}& {A}_{jk}(\x,\y; \delta)\,\rho_k(\y)\,{\rm d}S(\y) - \int_{\partial \Omega} {A}_{jk}(\x,\y;0)\,\rho_k(\y)\,{\rm d}S(\y)\|\\
&\le \|\int_{\partial \Omega} |{A}_{jk}(\x,\y;\delta)-{A}_{jk}(\x,\y;0)|\,|\rho_k(\y)|\,{\rm d}S(\y) \|\\
& \le \| \int_{\partial \Omega} \chi_{\epsilon}(|\x-\y|)|{A}_{jk}(\x,\y;\delta)-{A}_{jk}(\x,\y;0)|\,|\rho_k(\y)|\,{\rm d}S(\y) \|\\
& \quad + \| \int_{\partial \Omega} (1-\chi_{\epsilon}(|\x-\y|))|{A}_{jk}(\x,\y;\delta)-{A}_{jk}(\x,\y;0)|\,|\rho_k(\y)|\,{\rm d}S(\y) \|.
\end{align*}
Here $\| \cdot \|$ denotes the norm on $L^2(\partial \Omega; \mathbb{R}^3).$ Inserting our previous bound on ${A}$ into the above inequality yields
\begin{align*}
\|\int_{\partial \Omega}& {A}_{jk}(\x,\y; \delta)\,\rho_k(\y)\,{\rm d}S(\y) - \int_{\partial \Omega} {A}_{jk}(\x,\y;0)\,\rho_k(\y)\,{\rm d}S(\y)\|\\
& \le \| \int_{\partial \Omega} \chi_{\epsilon}(\x-\y)\frac{C}{|\x-\y|} |\rho_k(\y)|\,{\rm d}S(\y)\| + \| \int_{\partial \Omega} \frac{C \delta}{\epsilon^2} \| \rho(\y)\|_{\ell^2}\,{\rm d}S(\y)\|,
\end{align*}
where $\|\rho(\y)\|_{\ell^2}= \sqrt{\rho_1^2(\y)+\rho_2^2(\y)+\rho_3^2(\y)}$ is the standard Euclidean norm. The smoothness of $\partial \Omega$ implies that there is an overall constant $C$ such that
$$\int_{\partial \Omega} \chi_{\epsilon}(|\x-\y|)\frac{1}{|\x-\y|}\,{\rm d}S(\y) \le C \epsilon,$$
for all $\x \in \partial \Omega,$ and 
$$\int_{\partial \Omega} \chi_{\epsilon}(|\x-\y|)\frac{1}{|\x-\y|}\,{\rm d}S(\x) \le C \epsilon,$$
for all $\y \in \partial \Omega,$ and $\epsilon >0.$ It follows immediately from the Schur test~\cite{schur1911} that
$$\| \int_{\partial \Omega} \chi_{\epsilon}(|\x-\y|)\frac{1}{|\x-\y|} \|\rho(\y)\|_{\ell^2}\,{\rm d}S(\y)\| \le C\epsilon \|\rho\|.$$

We also observe that 
$$ \| \int_{\partial \Omega} \frac{ \delta}{\epsilon^2} | \rho(\y)|\,{\rm d}S(\y)\|\le \frac{C\delta}{\epsilon^2} |\partial \Omega|^{1/2} \|\rho\|.$$
Hence
\begin{align*}
\|\int_{\partial \Omega}& {A}_{jk}(\x,\y; \delta)\,\rho(\y)\,{\rm d}S(\y) - \int_{\partial \Omega} {A}_{jk}(\x,\y;0)\,\rho(\y)\,{\rm d}S(\y)\|\\
& \le C\left( \epsilon + \frac{\delta}{\epsilon^2}\right) \|\rho\|,
\end{align*}
for some constant $C$ independent of $\delta,\epsilon$ and $\rho.$ Choosing $\epsilon = \delta^{1/3}$ in the above inequality gives
\begin{align*}
\|\int_{\partial \Omega}& {A}_{jk}(\x,\y; \delta)\,\rho_k(\y)\,{\rm d}S(\y) - \int_{\partial \Omega} {A}_{jk}(\x,\y;0)\,\rho_k(\y)\,{\rm d}S(\y)\|\\
& \le 2C \delta^{1/3} \|\rho\|.
\end{align*}
In particular, the integral operator with kernel ${A}(\x,\y;0)$ is the norm limit of the compact operators with kernels ${A}(\x,\y;\delta),$ and hence is compact.

\end{proof}

\begin{remark}
    Up to rotations and translations, the functions $G^{B,2}$ and $G^{B,3}$ defined above are the Boussinesq-Cerruti solutions \cite{boussinesq1885application,cerruti1882ricerche} for the half-plane problem. In \cite{mindlin} Mindlin derived a more general formula for the Green's function in a halfspace with zero normal boundary stress. Taking the limit as the source approaches the boundary recovers $G^{B,2}$ and $G^{B,3}.$
\end{remark}

\section{String kernel representations}\label{sec:string}

Intuitively, the functions $G^{B,2}$ and $G^{B,3}$ defined in the previous section are ideal candidates from which to construct solutions to equations (\ref{eqn:elasto}) and (\ref{eqn:bc}). Indeed, one is tempted to introduce the ansatzes
\begin{align*}
u_{2}(\x) &= \int_{\partial \Omega} G^{B,2}(\x,\y;\n(\y))\,\rho_2(\y)\,{\rm d}S_{\y},\\
u_{3}(\x) &= \int_{\partial \Omega} G^{B,3}(\x,\y;\n(\y))\,\rho_3(\y)\,{\rm d}S_{\y},
\end{align*}
for the two-dimensional and three-dimensional problems, respectively, where $\rho_{2},\rho_3$ are unknown densities on the boundary $\partial \Omega,$ and the vectors $\n$ in the functions $G^{B,2}$ and $G^{B,3}$ are chosen to point in the direction of the outward facing normal of $\Omega$ at $\y.$ For exterior problems the directions of $\n$ should be flipped. 

Unfortunately, the branch cuts of $G^{B,2}$ and the line singularities of $G^{B,3}$ mean that for any exterior problem, or interior problems on sufficiently non-convex domains, the functions $u_2$ and $u_3$ will not satisfy the PDE $L[u] = 0.$ In this section we construct new functions $K^{B,2}$ and $K^{B,3}$ which have the same singularities as $G^{B,2}$ and $G^{B,3}$ near $\y,$ but possess singularities or discontinuities only on the \emph{finite} line segment $\ell_{\y,\bv,h}:=\{\x = \y + t \bv\,|\, t \in [0,h]\},$ where $h>0$ is a free parameter. The construction is based on the observation (see \cite{gimbutasFastMultipoleMethod2016} for example), that certain derivatives of $G^{B,2}$ and $G^{B,3}$ do not exhibit these line singularities. Appropriate kernels can thus be constructed by taking definite integrals over finite line segments. In two dimensions this is given by the following lemma which follows immediately from the definition of $G^{B,2}.$
\begin{lemma}
The function 
\begin{align*}{-}\bv \cdot \nabla_\y G^{B,2}(\x,\y;\bv)&=\frac{1}{2\pi \mu} \left[-\frac{\r\cdot\bv}{\alpha r^2}{\rm I}_2 + \frac{\bv \r^T+\r \bv^T}{r^2}-{2}\frac{(\r \cdot \bv)\r \r^T}{r^4}  \right]\\
&\quad +\frac{1}{2\pi \mu} \frac{1-\alpha}{\alpha} \frac{\bv^\perp \cdot \r}{r^2}\begin{pmatrix} \phantom{-}0 &1 \\ -1 & 0 \end{pmatrix},
\end{align*}
has a removable singularity on the line $\{\x = \y + t \bv\,|\, t \in (0,\infty)\},$ and satisfies $L[\bv \cdot \nabla_\y G^{B,2}](\x) = 0$ for all $\x \neq \y.$     
\end{lemma}
The following corollary follows immediately from the previous lemma, and gives the construction of a suitable function $K^{B,2}$ having the desired properties.
\begin{corollary}\label{col:g2}
Defining $K^{B,2}$ by
\begin{align*}
    K^{B,2}(\x,\y;\bv,h) &= -\int_{0}^h \bv \cdot \nabla_\y G^{B,2}(\x,\y+t \bv;\bv)\,{\rm d}t\\[5pt]
    & = G^{B,2}(\x,\y;\bv) - G^{B,2}(\x,\y+h\bv;\bv),
\end{align*}
gives a function which satisfies $L[K^{B,2}]=0$ everywhere except the line segment $\ell_{\y,\bv,h}.$ Moreover, $K^{B,2}(\x,\y;\bv,h)-G^{B,2}(\x,\y;\bv)$ is a smooth function of $\x$ in a neighborhood of $\y.$
\end{corollary}

Similar results hold in three dimensions, though an extra derivative is required.
\begin{lemma}
The function 
\begin{align}
\bv \cdot \nabla_\y (\bv \cdot \nabla_\y G^{B,3})(\x,\y;\bv) &=\frac{1}{4 \pi \mu} \bv \cdot \nabla_\y \bv \cdot \nabla_\y \left[\frac{{\rm I}_3}{r} + \frac{\r \r^T}{r^3} \right] \nonumber\\
&\quad + \frac{1}{4 \pi \mu} \frac{1-\alpha}{\alpha} \left[\frac{{\rm I}_3}{r^3} -2 \frac{\bv\bv^T}{r^3}  -3\frac{Q_\bv \r \r^TQ_\bv}{r^5} \right.\nonumber\\
&\quad  \left. -3  (\r \cdot \bv) \frac{\r \bv^T - \bv \r^T-(\r\cdot\bv) \bv\bv^T}{r^5}\right]\label{eqn:G3nn},
\end{align}
has a removable singularity on the line $\{\x = \y + t \bv\,|\, t \in (0,\infty)\},$ and satisfies $L[\bv \cdot \nabla_\y\,\bv \cdot \nabla_\y G^{B,3}](\x) = 0$ for all $\x \neq \y.$     
\end{lemma}
As for the two-dimensional case, the above lemma motivates the definition of a function, $K^{B,3},$ defined by integration of (\ref{eqn:G3nn}) over a finite line segment.
\begin{corollary}\label{col:g3}
Let $K^{B,3}$ be the function defined by
\begin{align}
\nonumber K^{B,3}(\x,\y;\bv,h) &= \int_{0}^h \int_{s}^h \bv \cdot \nabla_\y(\bv \cdot \nabla_\y G^{B,3}(\x,\y+t \bv;\bv))\,{\rm d}t\,{\rm d}s \\
&= \int_{0}^h  t\, \bv \cdot \nabla_\y(\bv \cdot \nabla_\y G^{B,3}(\x,\y+t \bv;\bv))\,{\rm d}t \label{eqn:Km3}\\
&= G^{B,3}(\x,\y;\bv) - G^{B,3}(\x,\y+h\bv;\bv)+h\, \bv \cdot \nabla_{\y}G^{B,3}(\x,\y+h \bv;\bv). \label{eqn:Km3_red}
\end{align}
Then $L[K^{B,3}]=0$ everywhere except the line segment $\ell_{\y,\bv,h}$ and $K^{B,3}(\x,\y;\bv,h)-G^{B,3}(\x,\y;\bv)$ is a smooth function of $\x$ in a neighborhood of $\y.$
\end{corollary}

It will also be convenient to have the expression for the stress tensors associated with $\bv \cdot \nabla_{\y} G^{B,3}$ and $\bv \cdot \nabla_{\y}\bv \cdot \nabla_{\y} G^{B,3}.$ Tedious but straightforward calculations yield
\begin{align*}
-\sigma[\bv \cdot \nabla_\y G^{B,3}(\cdot,\y,\bv)]_{ijk}(\x) &= -6\frac{\r_i\r_j \bv_k+\r_i \bv_j\r_k+\bv_i \r_j\r_k}{4 \pi r^5}+30\frac{\r_i \r_j \r_k (\r\cdot \bv)}{4 \pi r^7}\\
&+\frac{1}{4\pi }\frac{1-\alpha}{\alpha}\left\{-2\frac{(Q\r)_i(Q\r)_j \bv_k (R+2r)}{r^3R^3}\right.\\
&+ 2\frac{(Q\r)_i(Q\r)_j\r_k}{r^5R^3}(2r^2+3rR+3R^2)+2 \frac{Q_{i,j}\bv_k}{rR^2}\\
&-2\frac{R+r}{r^3R^2}(Q_{i,k}(Q \r)_j+Q_{j,k}(Q\r)_i+Q_{i,j}\r_k)\\
&\left.+2\frac{Q_{i,j}\bv_k}{r^3}-\frac{6 (\bv\cdot \r)\, Q_{i,j}\r_k}{r^5} \quad \right\},
\end{align*}
and
\begin{align*}
\sigma[\bv \cdot \nabla_\y\bv \cdot \nabla_\y &G^{B,3}(\cdot,\y,\bv)]_{ijk}(\x) = -12\frac{\r_i\bv_j \bv_k+\bv_i \r_j\bv_k+\bv_i \bv_j\r_k}{4 \pi r^5}+30\frac{\r_i \r_j \r_k }{4 \pi r^7}\\
&+ 60 \frac{(\r_i \r_j \bv_k + \r_i \bv_j \r_k + \bv_i\r_j\r_k)(\r\cdot \bv)}{4\pi r^7}-210 \frac{\r_i\r_j\r_k (\r\cdot\bv)^2}{4 \pi r^9}\\
&+\frac{1}{4\pi }\frac{1-\alpha}{\alpha}\left\{30 \frac{(Q \r)_i(Q \r)_j \r_k}{r^7}-6 \frac{Q_{i,k}(Q\r)_j+Q_{j,k}(Q\r)_i}{r^5}\right.\\
&\left.-12 \frac{Q_{i,j} \bv_k (\r\cdot \bv)}{r^5}-12 \frac{Q_{i,j} \r_k}{r^5}+30 \frac{Q_{i,j}\r_k \,(\r \cdot\bv)^2}{r^7} \right\},
\end{align*}
where, for ease of exposition, we have suppressed the dependence of $Q$ on $\bv.$

In the following, we define the functions $\Sigma^{B,2}$ and $\Sigma^{B,3}$ by 
$$\Sigma^{B,2}(\x,\y;h) := \n(\x) \cdot \sigma[K^{B,2}(\cdot,\y;\n(\y),h)](\x),$$ and 
$$\Sigma^{B,3}(\x,\y;h) := \n(\x) \cdot \sigma[K^{B,3}(\cdot,\y;\n(\y),h)](\x),$$
respectively, and choose $h = h(\y)$ to be a smooth non-vanishing function of $\y$ sufficiently small so that for each $\y \in \partial \Omega,$ the line segment $\ell_{\y,\n,h(\y)}$ only touches the surface $\partial \Omega$ at the single point $\y.$

We summarize the previous results in the following theorems.
\begin{theorem}\label{thm:string_rep23d}
    For a domain $\Omega$ with smooth boundary $\partial \Omega$ in $\mathbb{R}^d,$ $d=2,3,$ let $h:\partial \Omega \to \mathbb{R}^+$ be a smooth function such that for each $\y \in \partial \Omega,$ the intersection of the line segment $\ell_{\y,\n(\y),h(\y)}$ with $\partial \Omega$ is $\{\y\}.$ For $\rho \in L^2(\partial \Omega; \mathbb{R}^d),$
\begin{align*}
\lim_{\delta\to 0^+} \n_i(\x) \sigma&\left[\int K^{B,d}(\cdot,\y;\n(\y),h(\y))\rho(\y)\,{\rm d}S(\y)\right]_{ijk}(\x-\delta \n)\\
&=\rho_{j}(\x) + \int \Sigma^{B,d}_{jk}(\x,\y;h(\y)) \,\rho_k(\y)\,{\rm d}S(\y),
\end{align*}
in an $L^2$ sense. Moreover, the integral operator on the right-hand side is compact from $L^2(\partial \Omega;\mathbb{R}^d)$ to itself.
\end{theorem}
\begin{proof}
    {By construction, for $\delta >0$ sufficiently small, $K^{B,d}(\cdot,\y;\n(\y),h(\y))$ is smooth in a neighborhood of $\x-\delta \n(\x)$ for every $\x \in \partial \Omega$ and $\y \in \partial \Omega$ and hence we are free to interchange differentiation and integration.} From Corollary \ref{col:g2} in 2D and Corollary \ref{col:g3} in 3D, it is clear that $\sigma [K^{B,d}(\x,\y,\n(\y),h(\y))] - \sigma[G^{B,d}(\x,\y,\n(\y))]$ is smooth for $\x \in {\overline{\Omega}}$ sufficiently close to $\y \in \partial \Omega.$ Moreover, {since $\ell_{\y,\n(\y),h(\y)}$ meets $\overline{\Omega}$ only at $\y,$ $\sigma [K^{B,d}(\x,\y,\n(\y),h(\y))]$ and $T^{S,d}$} are smooth in {$\overline{\Omega}$} outside of any ball containing $\y.$ Next we write
    {\begin{align*}
    \sigma[K^{B,d}\,]&= 2\,T^{S,d} + W^{d} \chi_{{\beta}'}(|\x-\y|) + (\sigma[K^{B,d}] - 2T^{S,d}- W^{d} \chi_{{\beta}'}(|\x-\y|))\\
    &= 2T^{S,d} + W^d \chi_{{\beta}'}(|\x-\y|) + (\sigma[K^{B,d}] - \sigma[G^{B,d}] )\chi_{{\beta}'}(|\x-\y|)\\
    & \quad +(\sigma[K^{B,d}]- 2T^{S,d}) (1 - \chi_{\beta'}(|\x - \y|)) \\
&=2T^{S,d} + W^{d} \chi_{{\beta}'}(|\x-\y|)+ \mathcal{R}.
    \end{align*}}
    By construction, both terms comprising $\mathcal{R}$ are smooth in $\overline{\Omega}$ {uniformly in $\y \in \partial \Omega.$} It follows immediately that
    {\begin{align*}
    \lim_{\delta\to 0^+}  &\int \n(\x) \cdot\mathcal{R}(\x-\delta \n(\x),\y)\rho(\y)\,{\rm d}S(\y)\\
&=\int \n(\x)\cdot \mathcal{R}(\x,\y) \,\rho(\y)\,{\rm d}S(\y).
\end{align*}
    Applying Theorem \ref{thm:gs23_lim} to $T^{S,d}$ and Theorem \ref{thm:g2_lim} to $W^{d}$ gives the required result.}
\end{proof}
We collect the preceding results in the following theorem.
\begin{theorem}
\label{thm:main_res1}
Let $\Omega$ be a domain in $\mathbb{R}^d,$ $d=2,3$ with smooth boundary $\partial \Omega.$ For a right-hand side $f \in L^2(\partial \Omega; \mathbb{R}^d),$ if we introduce the ansatz
$$u_{j}(\x) = \int_{\partial \Omega} K^{B,d}_{jk}(\x,\y;\n(\y),h(\y))\,\rho_k(\y)\,{\rm d}S({\y}),$$
then the density $\rho$ satisfies the integral equation
\begin{align}
\rho_j(\x) + \int_{\partial \Omega} \Sigma^{B,d}_{jk}(\x,\y;h(\y))\rho_k(\y)\,{\rm d}S(\y) = f_j(\x).
\label{eqn:int_eq_23d}
\end{align}
Moreover, (\ref{eqn:int_eq_23d}) is a Fredholm integral equation of the second kind.
\end{theorem}

\begin{remark}
Recall that the traction boundary value problem in elastostatics has a finite-dimensional nullspace corresponding to rigid body motions. The dimensionality of the nullspace is $3$ in two dimensions, and $6$ in three dimensions, which is inherited by equation~{(}\ref{eqn:int_eq_23d}{)}. These nullspaces can be eliminated by enforcing that the density $\rho$ satisfy $3$ additional constraints in two dimensions and $6$ additional constraints in three dimensions, as is done for mobility and interior traction problems in Stokes flow, see~\cite{elastance-ref,corona2017integral}, for example.
\label{rem:null}
\end{remark}

\subsection{Limiting behavior of the integral equations}
Here we analyze the incompressible limit corresponding to taking $\lambda \to \infty$, or equivalently, $\alpha \to 1^-,$ with $\mu$ fixed. We start with the two-dimensional case. In this limit, for any $\x$ not on the half-line emanating from $\y$ in the direction $\n,$ we have
$\lim_{\alpha \to 1} G^{B,2}(\x,\y;\bv) =2\, G^{S,2}(\x,\y).$ Turning to the normal stress, {$W^2 \equiv 0,$ i.e. $\bv \cdot \sigma[G^{B,2}] = 2 \bv_i T^{S,2}_{ijk}$ for all $\alpha.$} In particular, the integral operator in the 2D integral equation (\ref{eqn:int_eq_23d}) is independent of $\alpha.$ Explicitly, (\ref{eqn:int_eq_23d}) can be written as
\begin{align}
\frac{1}{2}\rho_i(\x) &- \frac{1}{\pi}\int_{\partial \Omega}\n_j(\x)\frac{(x_i-y_i)(x_j-y_j)(x_k-y_k) }{|\x-\y|^4}\,\rho_k(\y)\,{\rm d}S(\y) \\
&= \frac{1}{2}f_i(\x) - \frac{1}{ \pi}\int_{\partial \Omega}\n_j(\x)\frac{(x_i-y_i')(x_j-y_j')(x_k-y_k') }{|\x-\y'|^4}\,\rho_k(\y)\,{\rm d}S(\y),
\end{align}
with $\y' = \y + h(\y) \n(\y).$ The left-hand side of the previous equation is just the integral equation for solving the traction boundary value problem when the velocity is represented using the single layer potential for Stokes flow in two dimensions, while the integral operator on the right-hand side has a smooth kernel and so is compact. Hence, for any $\alpha,$ (\ref{eqn:int_eq_23d}) is a compact perturbation of the surface  traction integral equation and moreover, if (\ref{eqn:int_eq_23d}) is solvable for one $\alpha$ then it is solvable for every $\alpha,$ (in appropriate spaces which handle the non-uniqueness corresponding to rigid body motions, see Remark~\ref{rem:null}). 

Next we turn to 3D. Suppose that $\x$ is not on the half-line emanating from $\y$ and pointing in the direction $\n.$ It follows immediately from the definition of $\Sigma^{B,3}$ and Theorem \ref{thm:string_rep23d} that
{\begin{align*}
    \Sigma^{B,3}_{jk}(\x,\y,h(\y)) & = 2 \n_i(\x)\left[T^{S,3}_{ijk}(\x,\y)-T^{S,3}_{ijk}(\x,\y')+h(\y)\n(\y)\cdot \nabla_{\y'}T^{S,3}_{ijk}(\x,\y') \right]\\[5pt]
    &  +(1-\alpha) \hat{K}_{j,k}(\x,\y),
\end{align*}}
where $\hat{K}_{j,k}$ is a kernel corresponding to a compact operator. In particular, (\ref{eqn:int_eq_23d}) is a compact perturbation of the integral equation arising in the traction boundary value problem when the velocity is represented using the Stokes layer potential, and in the limit as $\alpha \to 1^{-}$ the integral operator converges to{\begin{align}\label{eqn:3d_stokes_limit}
\rho_i(\x) &- \frac{3}{2\pi}\int_{\partial \Omega}\n_j(\x)\frac{(x_i-y_i)(x_j-y_j)(x_k-y_k) }{|\x-\y|^5}\,\rho_k(\y)\,{\rm d}S(\y)  \\
&= f_i(\x) - \frac{3}{ 2\pi }\int_{\partial \Omega}\n_j(\x)\frac{(x_i-y_i')(x_j-y_j')(x_k-y_k') }{|\x-\y'|^5}\,\rho_k(\y)\,{\rm d}S(\y)\nonumber\\
&\quad\quad +\frac{3}{ 2\pi }\int_{\partial \Omega}h(\y)\n_j(\x)\n_\ell(\y)\partial_{y'_\ell}\left[\frac{(x_i-y_i')(x_j-y_j')(x_k-y_k') }{|\x-\y'|^5}\right]\,\rho_k(\y)\,{\rm d}S(\y)\nonumber.
\end{align}}
In particular, if (\ref{eqn:3d_stokes_limit}) is invertible then (\ref{eqn:int_eq_23d}) will be invertible for all $\alpha$ sufficiently close to 1.

\section{Numerical apparatus}\label{sec:num_app}
In this section we briefly sketch relevant details for the discretization and solution of the integral equation (\ref{eqn:int_eq_23d}). A nice feature of our approach is that only minor modifications are required in order to use standard integral equations toolboxes, chunkIE \cite{chunkie} in two dimensions and fmm3DBIE \cite{fmm3dbie} in three dimensions. Moreover, the kernels used in both our integral representation and the corresponding integral equations can be applied quickly via a small extension to the standard fast multipole method toolboxes \cite{fmm3d}.

\subsection{Evaluation of the kernels}
In two dimensions, the kernel $\Sigma^{B,2}(\cdot,\y;h(\y))$ is only singular at $\y$ and $\y + h(\y) \n(\y).$ Indeed, combining the identity (\ref{eqn:sigg2}) with Corollary \ref{col:g2}, we obtain
$$\Sigma^{B,2}(\x,\y,h(\y)) = -2\frac{(\r\cdot \n(\x))r_jr_k}{\pi r^4}+2\frac{(\r'\cdot\n(\x))r'_jr'_k}{\pi (r')^4},$$
with $\r' = \x-\y -h \n(\y).$ By construction, for a given geometry we can choose $h(\y)$ so that $r'=|\r'|$ is bounded away from zero. In particular, the second term can be stably evaluated with minimal loss of accuracy due to catastrophic cancellations. The first term arises in other boundary integral equations, see \cite{pozrikidis1992boundary} for example, and can be evaluated using standard methods. Indeed, we observe that if $\partial \Omega$ is smooth, and $\r(t)$ is an arclength parameterization of $\partial \Omega,$ then $\r(t)-\r(s) = (t-s)\tau(t) + \frac{\kappa}{2}(t-s)^2 \n(t) +O((t-s)^3).$ Then, the first term in the kernel becomes
$$-\kappa \frac{r_jr_k}{\pi r^2} +O(|\r|),$$
for $\x,\y \in \partial \Omega.$ Using this approximation for $|\x-\y|$ small enough allows the kernel to be stably evaluated. In principle, higher order expansions can be used --- these involve evaluating higher order derivatives of the normal vector.

For $K^{B,2},$ similar reasoning applies. In principle, for $\x \approx \y + t \n(\y)$ with $t>0$ sufficiently large, one will get catastrophic cancellation involving the difference of the kernels. In practice, this will not be significant except when evaluating points very far away, where other formulas can be used.

In three dimensions, singularities are present on the line segments $\ell_{\y,\n,h}$ both in the integral equation kernel $\Sigma^{B,3}(\cdot,\y;h(\y))$ and the kernel used in the representation $K^{B,3}.$ In particular, if (\ref{eqn:Km3_red}) is used to calculate both $\Sigma^{B,3}$ and $K^{B,3}$ then when $\x \approx \y + t \n(\y)$ with $t>h(\y),$ catastrophic cancellations will occur, even though the kernels are smooth and decaying in $|\x-\y|$. To avoid this issue, for $|\x-\y|>h(\y)$  we instead use the integral representation of $K^{B,3}$ given in (\ref{eqn:Km3}). From (\ref{eqn:G3nn}) we note that the integrand is smooth for $\x$ away from $\ell_{\y,\n(\y),h(\y)}.$ It follows that the integral in (\ref{eqn:Km3}) may be computed using a standard Gauss-Legendre quadrature with an order which depends only logarithmically on the desired accuracy.

\subsection{Discretization of the integral equations} 
As mentioned above, the boundary integral equation (\ref{eqn:int_eq_23d}) is amenable to standard discretization and quadrature methods. Here we use the software package chunkIE \cite{chunkie} in two dimensions, and fmm3DBIE \cite{fmm3dbie} in three dimensions, which discretize and solve integral equations using a patch-based collocation scheme \cite{greengard2021fast}. In particular, both the geometry and unknowns are represented using piecewise polynomial expansions. For curves in two dimensions these quantities are interpolated from Gauss-Legendre quadrature nodes, while for surfaces in three dimensions they are interpolated from Vioreanu-Rokhlin nodes~\cite{vioreanu_2014,xiao2010numerical}. High-order quadratures for weakly-singular integral operators are constructed via \emph{generalized Gaussian quadrature} \cite{ggq1,ggq2,ggq3}.

\subsection{Fast algorithms} As is typical for boundary integral equations, discretizing (\ref{eqn:int_eq_23d}) yields a
dense linear system. In either case, however, both the application of these matrices and their inversion can be
accelerated using existing machinery with only minor modifications. Here we briefly sketch the necessary observations, deferring a complete 3D implementation to a subsequent paper.

In two dimensions, the kernel appearing in the integral equation differs from the traction of the 2D Stokes single-layer by only a smooth function (which can itself be written as the traction of the Stokes single-layer from a suitable source). Thus, only two Stokeslet evaluations are required per source. Existing fast
iterative solvers based on the 2D Stokes FMM~\cite{fmm2d}, and
fast direct solvers based on hierarchical
skeletonization~\cite{martinsson2005fast,greengard2009fast,sushnikova2023fmm,minden2017recursive,corona2015n},
therefore apply essentially unchanged.

In three dimensions, Corollary~\ref{col:g3} expresses each entry of
$\Sigma^{B,3}(\x,\y;h)$ as a line integral over $\ell_{\y,\n,h}$ of a
linear combination of terms
\begin{equation}\label{eqn:gen_form}
f(\x)\, g(\y)\, \partial_{\x}^{{\gamma}} \frac{1}{|\x - \y|},
\qquad |{\gamma}| \le 3,
\end{equation}
with $f, g$ explicit scalar functions. For $|\x - \y|$ large compared
to $h$, the line integral is captured to high precision by a
fixed-order Gauss–Legendre rule (order 16 in our experiments),
reducing each far-field interaction to a sum of terms of the form
(\ref{eqn:gen_form}). Kernels of this type can be applied quickly using the
standard 3D Laplace FMM~\cite{GREENGARD1987325,fmm3d}.
For $|\x - \y|$ comparable to or smaller than $h$, the kernel is
evaluated directly via (\ref{eqn:Km3_red}) together with standard
near-singular and self-interaction quadratures. This yields an $O(N)$
algorithm for applying the system matrix, provided  the strings are short
compared with the FMM leaf-box size, a condition enforceable by
sufficient tree refinement. An analogous constraint on string endpoints
applies to hierarchical direct solvers, where long strings can destroy
the low-rank structure of far-field blocks. Perhaps the most closely related
existing construction is~\cite{gimbutasFastMultipoleMethod2016}, which
treats the case of infinite, uniformly vertical strings.

\subsection{Heuristics for the choice of $h$} 
In two dimensions there is additional flexibility which greatly simplifies matters. For ease of exposition we focus on the interior problem for a simply connected domain $\Omega,$ though an almost identical approach can be used for exterior problems and non-simply connected $\Omega.$ For fixed $\x$ and $\y$ we observe that
$K^{B,2}(\x,\y;\n,h)$ has a branch cut $c$ along $\ell_{\y,\n,h}$ connecting $\y$ to $\y_s=\y+h\n.$ The remainder of the kernel is smooth in $\x$ except at $\y$ and $\y_s.$ 

Clearly, one is free to continuously deform $c$ as long as it does not pass through $\x.$ Let $\y_*$ be a point in $\Omega^c$ connected to $\y$ and $\y_s$ by smooth curves $c',c'' \subset \overline{\Omega^c},$ respectively, such that $c'\cup c''$ is continuously deformable to $\ell_{\y,\n,h}$ while remaining in $\Omega^c.$ Then 
\begin{align*}
K^{B,2}(\x,\y;\n,h) &= \big(G^{B,2}(\x,\y;c')-G^{B,2}(\x,\y_*;c')\big)\\[5pt]
&-\big(G^{B,2}(\x,\y_*;c'')-G^{B,2}(\x,\y_s;c'')\big),
\end{align*}
where the branch cut in the first term lies along $c'$ and the branch cut in the second term lies along $c''.$ Here, with some abuse of notation, we replace the vector $\n$ in the arguments of $G^{B,2}$ with the curve along which the branch cut lies.

We note that the second term satisfies the homogeneous PDE in $\Omega,$ and has smooth normal stress on $\partial \Omega.$ Hence, we are free to replace $K^{B,2}$ in our representation with 
$$K^{B,2}_{\y_*}:=G^{B,2}(\x,\y;c')-G^{B,2}(\x,\y_*;c'),$$ 
where the branch cut in $G^{B,2}$ lies along $c'$ {(see Figure~\ref{fig:Kystar}).}
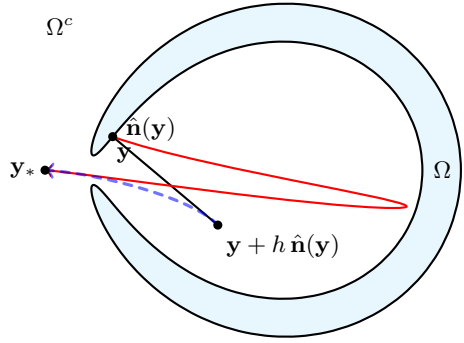
\begin{figure}[h]
\centering
\begin{tikzpicture}[scale=1.0,line cap=round]
  \def\cavpath{plot[smooth cycle,tension=0.55] coordinates {
    (1.895,0.000)  (1.861,0.376)  (1.762,0.740)  (1.601,1.081)
    (1.384,1.389)  (1.117,1.655)  (0.811,1.872)  (0.476,2.036)
    (0.121,2.144)  (-0.243,2.197) (-0.605,2.196) (-0.957,2.146)
    (-1.290,2.051) (-1.600,1.920) (-1.880,1.759) (-2.127,1.578)
    (-2.341,1.384) (-2.521,1.186) (-2.667,0.991) (-2.783,0.807)
    (-2.870,0.640) (-2.931,0.495) (-2.969,0.375) (-2.986,0.284)
    (-2.986,0.225) (-2.969,0.197) (-2.938,0.201) (-2.891,0.236)
    (-2.829,0.300) (-2.751,0.390) (-2.656,0.503) (-2.542,0.634)
    (-2.408,0.779) (-2.252,0.931) (-2.073,1.086) (-1.872,1.236)
    (-1.648,1.376) (-1.403,1.498) (-1.140,1.596) (-0.862,1.664)
    (-0.574,1.698) (-0.282,1.694) (0.008,1.649)  (0.288,1.562)
    (0.550,1.433)  (0.788,1.264)  (0.994,1.059)  (1.161,0.824)
    (1.285,0.563)  (1.360,0.286)  (1.386,0.000)  (1.360,-0.286)
    (1.285,-0.563) (1.161,-0.824) (0.994,-1.059) (0.788,-1.264)
    (0.550,-1.433) (0.288,-1.562) (0.008,-1.649) (-0.282,-1.694)
    (-0.574,-1.698)(-0.862,-1.664)(-1.140,-1.596)(-1.403,-1.498)
    (-1.648,-1.376)(-1.872,-1.236)(-2.073,-1.086)(-2.252,-0.931)
    (-2.408,-0.779)(-2.542,-0.634)(-2.656,-0.503)(-2.751,-0.390)
    (-2.829,-0.300)(-2.891,-0.236)(-2.938,-0.201)(-2.969,-0.197)
    (-2.986,-0.225)(-2.986,-0.284)(-2.969,-0.375)(-2.931,-0.495)
    (-2.870,-0.640)(-2.783,-0.807)(-2.667,-0.991)(-2.521,-1.186)
    (-2.341,-1.384)(-2.127,-1.578)(-1.880,-1.759)(-1.600,-1.920)
    (-1.290,-2.051)(-0.957,-2.146)(-0.605,-2.196)(-0.243,-2.197)
    (0.121,-2.144) (0.476,-2.036) (0.811,-1.872) (1.117,-1.655)
    (1.384,-1.389) (1.601,-1.081) (1.762,-0.740) (1.861,-0.376)
  }}
  \fill[cyan!8] \cavpath;
  \draw[thick]  \cavpath;

  \coordinate (y)  at (-2.706, 0.444);
  \coordinate (yh) at (-1.317,-0.726);
  \coordinate (ys) at (-3.600, 0.000);

  \draw[thick] (y) -- (yh);

  \draw[red,thick,->]
    (y)        .. controls (-1.5, 0.0) and (2.0,-0.5) .. (1.0,-0.5)
               .. controls (0.0,-0.5) and (-3.0,-0.05) .. (ys);

  \draw[blue,very thick,dashed,opacity=0.55,->]
    (yh) .. controls (-1.6,-0.4) and (-2.8,-0.05) .. (ys);

  \fill (y)  circle (0.06);
  \fill (yh) circle (0.06);
  \fill (ys) circle (0.06);

  \node[font=\small] at (-2.55,0.20)  {$\y$};
  \node[font=\small] at (-2.20,0.55)  {$\n(\y)$};
  \node[anchor=north west,font=\small,inner sep=3pt] at (yh)
        {$\y + h\,\n(\y)$};
  \node[anchor=east,font=\small,inner sep=3pt] at (ys) {$\y_*$};

  \node[font=\small] at (1.65, 0.0) {$\Omega$};
  \node[font=\small] at (-3.4, 1.9) {$\Omega^c$};
\end{tikzpicture}
\caption{The contour-deformation kernel $K^{B,2}_{\y_*}$ on a smooth
cavity domain $\Omega.$ The black
segment is the original truncated string from $\y$ to $\y + h\,\n(\y)$,
on which $K^{B,2}(\cdot,\y;\n,h)$ has its line singularity. The solid
red contour is a smooth deformation from $\y$ to a point
$\y_* \in \Omega^c$, sweeping through the cavity and out the opening
of $\Omega$; this defines the modified kernel
$K^{B,2}_{\y_*}(\x,\y) = G^{B,2}(\x,\y;\n(\y)) -
G^{B,2}(\x,\y_*;\n(\y))$. The dashed contour from $\y + h\,\n(\y)$ to
$\y_*$ is the auxiliary curve used to identify the difference
$K^{B,2}_{\y_*} - K^{B,2}$ as an integral that is regular in $\Omega$.}
\label{fig:Kystar}
\end{figure}

In three dimensions, the optimal choice of $h$ remains something of an open problem. On one hand, if $h$ is chosen too small, then the singularities from each end of the string almost cancel, and the condition number deteriorates. On the other hand, if $h$ is too large then the singularities associated with the far end of the string will lead to nearly singular interactions with other parts of the boundary, requiring finer discretization and possibly again increasing the condition number. 
In all of the examples considered in this work, we use a fixed string length.

\section{Numerical illustrations}\label{sec:num_ill}
To illustrate the performance of our approach, we consider the solution of the PDE in the interior of a star-shaped geometry ($\Omegastar$), and a cavity-like ($\Omegacavity$) domain in two dimensions, and an ellipsoid ($\Omegae$), and a star-shaped torus ($\Omegat$) in three dimensions. The parametrization of the star-shaped geometry denoted by $\gammastar$ is given by
$$
\gammastar(t) = \big(1 + 0.3\cos{(5t)}\big)\begin{bmatrix}\cos{(t)} \\ \sin{(t)}\end{bmatrix} \,, \quad t\in[0,2\pi) \,. 
$$
For the cavity domain, suppose $z(t) = x(t) + iy(t)$, $t\in [0,2\pi)$ is an ellipse centered at $(20,0)$ with semi-major and minor axis $(a,b)$. Then $$\gammacavity(t) = \frac{z^{2\zeta}(t)}{\max_{t\in [0,2\pi)}(|z(t)|^{2\zeta})}\,,$$ with $(a,b,\zeta) = (0.397,8.02,3.965)$. The ellipsoid in three dimensions is coordinate-axis aligned with principal semi-axes $(5.1,1.0,2.0)$ respectively. Finally, the parametrization for the star-shaped torus denoted by $\gammat$ is given by
$$
\gammat(u,v) = \begin{bmatrix}
\big(R_{1} + R_{3} \cos{(3v)} + R_{2}\cos{(u)}\big) \cos{(v)} \\
\big(R_{1} + R_{3} \cos{(3v)} + R_{2}\cos{(u)}\big) \sin{(v)} \\
R_{2} \sin{(u)}  
\end{bmatrix}\, ,
$$
with $(u,v) \in [0,2\pi)^2$, $R_{1} = 4$, $R_{2} = 1.5$, and $R_{3} = 0.25$. We plot the star-shaped domain, the cavity and the star-shaped torus in Figure~\ref{fig:geom}. Unless stated otherwise, the string length $h$ is $0.1$ for examples in two dimensions, and $0.25$ for the examples in three dimensions.

\begin{figure}[!htbp]
\begin{center}
\includegraphics[width=0.7\linewidth]{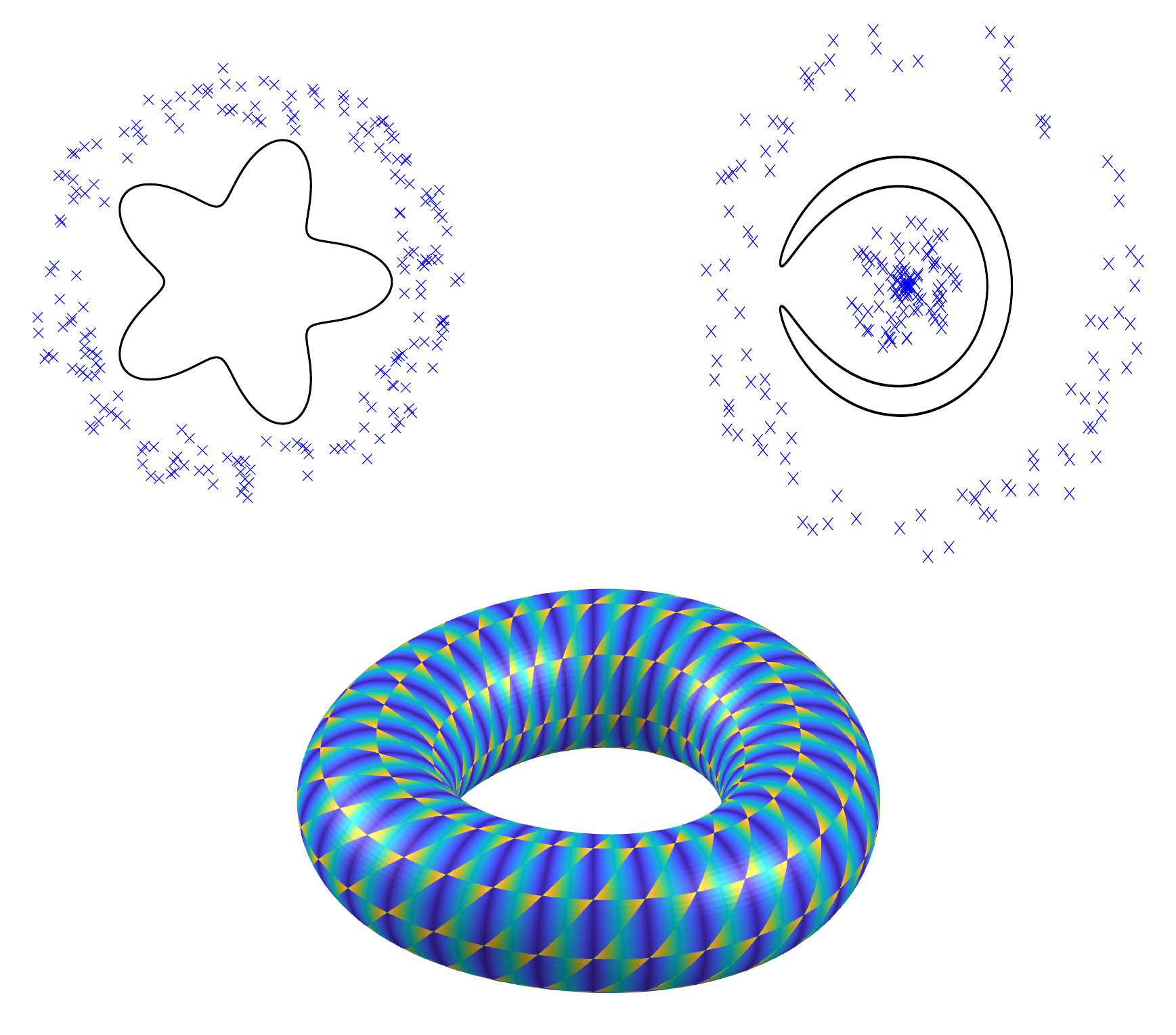}
\caption{Top row: star-shaped domain and cavity domain in two dimensions. The blue crosses indicate the source locations $\y_j$ for generating the exact solution. Bottom row: star-shaped torus domain.}
\label{fig:geom}
\end{center}
\end{figure}

We test the accuracy of the string kernel representation using a known solution in the domain. In two dimensions, suppose $\uex(\x) = \sum_{j=1}^{B} G^{2}(\x,\y_{j}) \sigma_{j}$, where the locations of the $\y_{j}$ are denoted by crosses in Figure~\ref{fig:geom}, the components of $\sigma_{j} \in \mathbb{R}^{2}$ are drawn uniformly in {$[-1,1]$}, $G^{2}(\x,\y)$ is the elastostatic Green's function in two dimensions given by
$$
G^{2}(\x,\y) = \frac{1}{4\pi \mu}\left(-(2-\alpha) \log{|\x - \y|} I_{2} + \alpha\frac{(\x - \y)(\x - \y)^{T}}{|\x- \y|^2}\right)\,,
$$
$\mu = 1$, and $\lambda = 10$. We then represent the solution and solve the integral equation as described in Theorem~\ref{thm:main_res1}. The integral equation is discretized using chunkIE, and the solution is computed at a grid of targets inside the domain denoted by $\x_i$, $i=1,2,\ldots,N_{t}$. For the star-shaped domain, $N_{t} =43410$, and for the cavity $N_{t} = 14764$. Since the solution is only unique up to a rigid body motion, we first find the best $v_{0} \in \mathbb{R}^{2}, \omega \in\mathbb{R}$, which minimize the residual
$$ v_{0}^{\star}, \omega^{\star} = \textrm{argmin}_{v_{0}, \omega} \sum_{i=1}^{N_{t}} \left| \uex(\x_{i}) - u(\x_{i}) - v_{0} - \omega (\x_{i}-\x_{c})^{\perp} \right|^2\,,$$
where $\x_{c}$ is the centroid of the domain given by $\x_{c} = {|\partial \Omega|^{-1}}\int_{\partial \Omega} \x \,{\rm d}S(\x)$, and $\x^{\perp} = (-x_{2}, x_{1})$. Figure~\ref{fig:err2d} plots the relative pointwise residual $r(\x) = \|\uex(\x) - u(\x) - v^{\star}_{0} - \omega^{\star}(\x - \x_c)^{\perp}\|_{\ell^2}/\|\uex(\x)\|_{\ell^2},$ along with the computed solution $u(\x_i)$ for both the domains.

\begin{figure}[!htbp]
\begin{center}
\includegraphics[width=\linewidth]{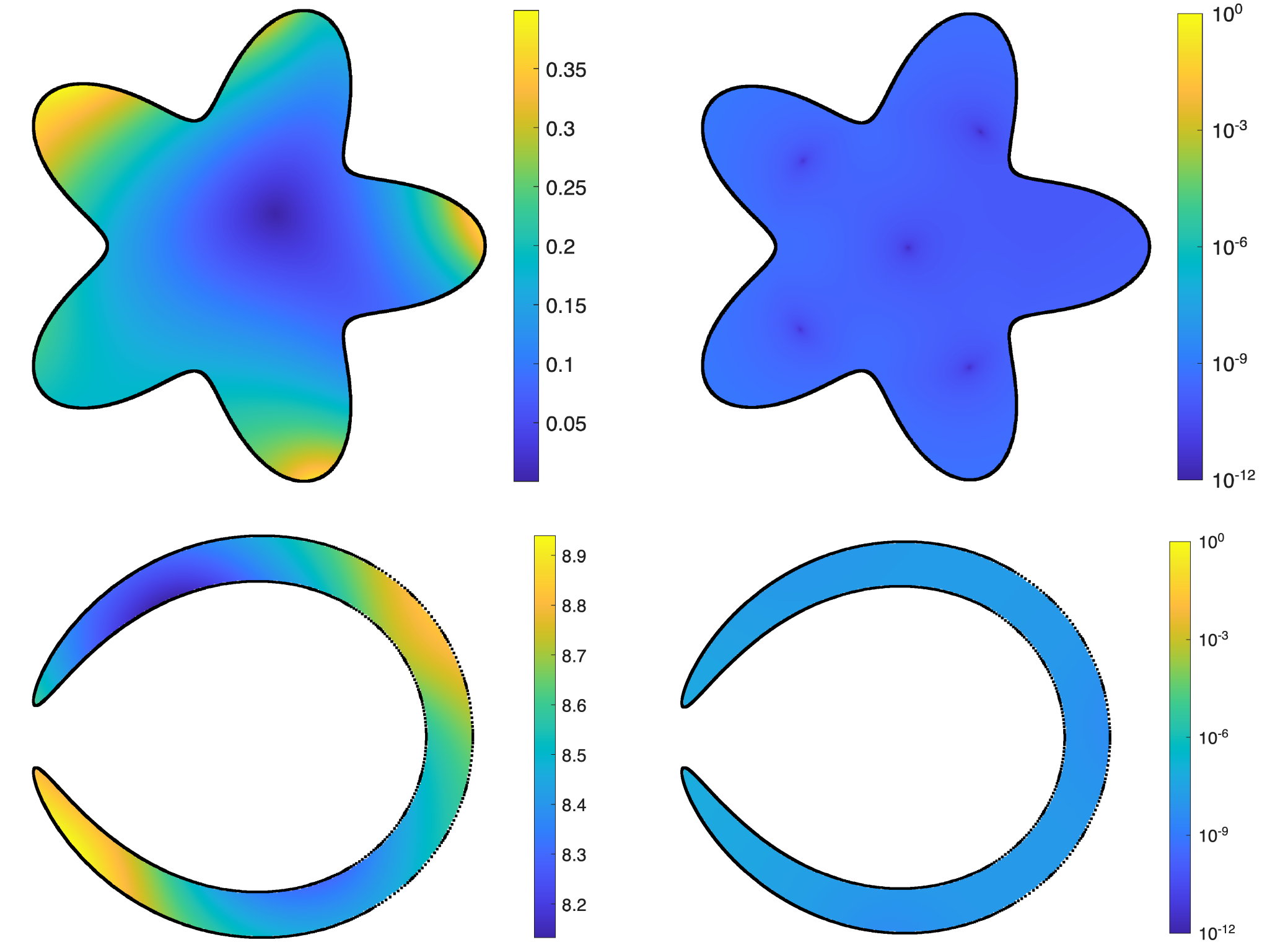}
\caption{Left column: computed solutions $u(\x)$ on the star-shaped and cavity domains. Right column: Relative residual $r(\x)$.}
\label{fig:err2d}
\end{center}
\end{figure}

In three dimensions, suppose $\uex(\x) = G^{3}(\x,\y_{0}) \sigma_{0}$, where $\y_{0} = (3.1,3.19,10.19),$ {$\sigma_{0} = 10^{3} \times (2.2,4.1,0.6)$}, $G^{3}(\x,\y)$ is the elastostatic Green's function in three dimensions given by
$$G^{3}(\x,\y) = \frac{1}{{8}\pi \mu}\left(\frac{(2-\alpha)}{|\x - \y|}I_{3} + \alpha \frac{(\x - \y)(\x - \y)^{T}}{|\x - \y|^3} \right),$$
$\lambda = 1$, and $\mu = 1$. 
Following a procedure similar to the one discussed above, we test the accuracy of our solver in three dimensions. The linear system is formed densely and solved using GMRES until a relative residual of $10^{-10}$ is reached. In Table~\ref{tab:convergence}, we illustrate high-order convergence, when the geometries were discretized using $5$th order patches and $7$th order patches. The reported error is the $\ell^{2}$ norm of the residual $r(\x_i) = \uex(\x_i) - u(\x_i) - v^{\star}_{0} - \omega^{\star} \times (\x_i - \x_c)$ at $200$ random targets $\x_i$ inside the domain, relative to the norm of the boundary data, $\sigma[\uex]$. In all cases, the number of GMRES iterations required is independent of the mesh size, indicating a mesh-independent condition number. In all examples, we eventually observe the expected order of convergence. 
{\begin{remark}
The iteration counts reported in Table~\ref{tab:convergence} are larger than one might
expect for a second-kind integral equation on a smooth domain. Their
independence of the mesh reflects the fact that the operator is a compact
perturbation of the identity; the magnitude of the count is governed instead by
the spectral spread of that compact perturbation, which is bounded independently
of $N$ but is not small for the parameter values used here. The high iteration count can in part be attributed to the choice of $h$, since as $h \to 0$ the representation
degenerates to one with a hypersingular kernel and the equation loses its
second-kind character.
\end{remark}
}

\begin{table}[ht]
\centering
\begin{tabular}{lccccc}
\hline
Geometry & $N_{\textrm{patches}}$ & $n_{\textrm{order}}$ & $n_{\textrm{iter}}$ & $\displaystyle\frac{\|r\|_{\ell^2} }{\|\sigma[\uex]\|_{L^2}}$ & \begin{tabular}{@{}c@{}}Empirical  \\ convergence \\order\end{tabular} \\
\hline
Ellipsoid &  56 & 5 & 211 & $2.7 \times 10^{-2}$ &        \\
Ellipsoid & 224 & 5 & 206 & $2.4 \times 10^{-2}$ & $0.15$ \\
Ellipsoid & 504 & 5 & 190 & $7.7 \times 10^{-4}$ & $8.49$ \\
Ellipsoid & 896 & 5 & 187 & $2.0 \times 10^{-4}$ & $4.66$ \\
\hline
Ellipsoid &  56 & 7 & 220 & $2.5 \times 10^{-2}$ &        \\
Ellipsoid & 224 & 7 & 190 & $1.5 \times 10^{-3}$ & $4.03$ \\
Ellipsoid & 504 & 7 & 187 & $3.9 \times 10^{-5}$ & $9.04$ \\
Ellipsoid & 896 & 7 & 186 & $2.6 \times 10^{-6}$ & $9.39$ \\
\hline
\hline
Star-shaped torus &  128 & 5 & 284 & $3.6 \times 10^{-2}$ &        \\
Star-shaped torus &  512 & 5 & 271 & $2.6 \times 10^{-3}$ & $3.83$ \\
Star-shaped torus & 1152 & 5 & 260 & $5.8 \times 10^{-4}$ & $3.68$ \\
Star-shaped torus & 2048 & 5 & 257 & $1.0 \times 10^{-4}$ & $6.05$ \\
\hline
Star-shaped torus &  128 & 7 & 286 & $4.7 \times 10^{-3}$ &         \\
Star-shaped torus &  512 & 7 & 263 & $2.4 \times 10^{-4}$ &  $4.27$ \\
Star-shaped torus & 1152 & 7 & 257 & $2.5 \times 10^{-5}$ &  $5.62$ \\
Star-shaped torus & 2048 & 7 & 253 & $1.3 \times 10^{-6}$ & $10.16$ \\
\hline
\end{tabular}
\caption{Convergence results for the ellipsoid and star-shaped torus in three dimensions. $N_{\textrm{patches}}$ is the number of patches in the discretization, $n_{\textrm{order}}$ is the order of discretization on each patch, $n_{\textrm{iter}}$ is the number of GMRES iterations required to converge to a relative residual of $10^{-10}$, and $\|r\|_{\ell^2}/\|\sigma[\uex]\|_{L^2(\partial \Omega;\mathbb{R}^d)}$ is a measure of the relative error. Finally, in the last column, we give the implied order of convergence.}
\label{tab:convergence}
\end{table}

Next, we demonstrate the stability of the string kernel integral equation formulation in the Stokes limit. In two dimensions, the condition number is trivially independent of $\lambda$ since the kernel of the integral equation, $\sigma[K^{B,2}]$ does not depend on $\lambda$. In three dimensions, we illustrate the behavior in this limit numerically by computing the condition numbers of the discretized linear systems for an ellipsoid with principal semi-axes $2,1,2$ as a function of $\lambda \in (10^{-5}, 10^{8})$ ($\lambda \to \infty$ corresponds to the Stokes limit), see Figure~\ref{fig:cond}. The discrete linear systems were $\ell^{2}$ scaled to obtain a closer approximation of the underlying physical condition number~\cite{ggq1}. We observe that the condition number for our representations is practically independent of $\lambda$. The ratio of the maximum to minimum condition number over the range of $\lambda$ is $1.63$.

\begin{figure}[!htbp]
\begin{center}
\includegraphics[width=0.7\linewidth]{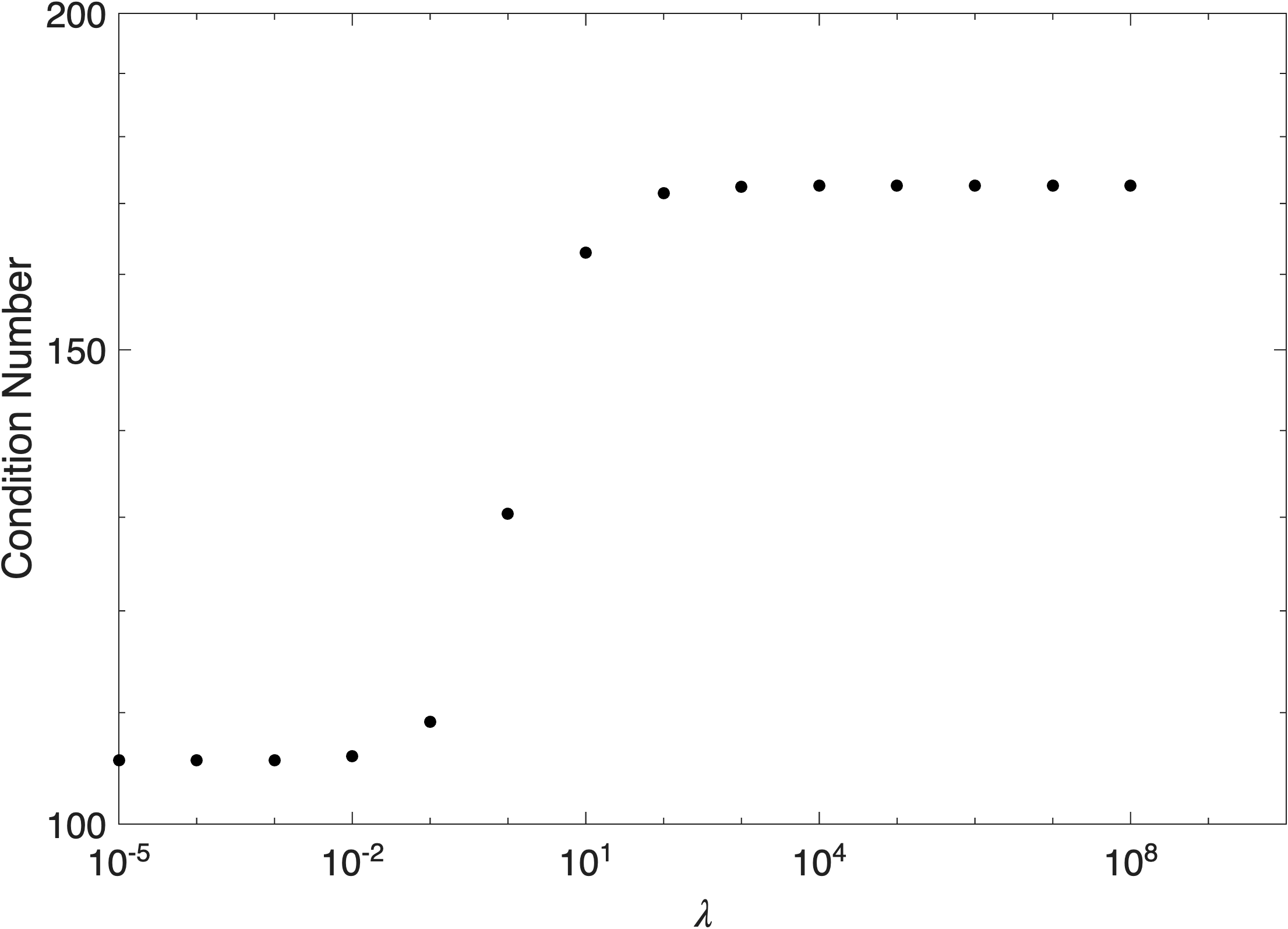}
\caption{Condition number as a function of $\lambda$ for an ellipsoid with principal semi-axes $(2,1,2)$ in three dimensions.}
\label{fig:cond}
\end{center}
\end{figure}

{
\begin{remark}
The Lam\'e parameters enter the integral operators only through $\alpha$, and
$\alpha \to 1^{-}$ both as $\mu \to 0$ with $\lambda$ fixed and as
$\lambda \to \infty$ with $\mu$ fixed. The behavior of the integral equation is
therefore the same in the two limits, though the underlying boundary value
problems are not. When $\lambda \to \infty$ with $\mu$ fixed, the bulk modulus
diverges, $\nabla \cdot u \to 0$, and $u$ itself converges to the velocity field
of a Stokes traction boundary value problem with viscosity $\mu$. When
$\mu \to 0$ with $\lambda$ fixed, the bulk modulus remains bounded and
$\nabla \cdot u$ does not vanish; instead the strain grows like
$\mu^{-1}$, and it is the rescaled displacement $w = \mu u$ that converges to
the solution of a Stokes traction boundary value problem with unit viscosity.
\end{remark}
}

The condition number is sensitive to both the choice of the string length $h$, and the orientation vector $\bv$. In Figure~\ref{fig:cond_h}, we plot the condition number as a function of $h$ on the ellipsoid with principal semi-axes (2,1,2) with $\bv = \n$, and on $\Omegastar$ with $\bv = \n$ and $\bv = \hat{\mathbf{r}}$, where $\hat{\mathbf{r}} = (x_{1}, x_{2})/\sqrt{x_{1}^2 + x_{2}^2}$ is the radial vector. We see that in both cases, the condition number tends to $\infty$ as the string length $h\to 0$. Moreover, radially oriented strings outperform normally oriented strings for star-shaped domains. In either case, the condition number tends to improve as the string length increases. The oscillatory nature of the condition number for normally oriented strings is specific to the star-shaped domain owing to the oscillatory nature of the distance of the closest point to the boundary as $h$ is increased. In both of the examples the same discretization was used for all $h$.

\begin{figure}[!htbp]
\begin{center}
\includegraphics[width=\linewidth]{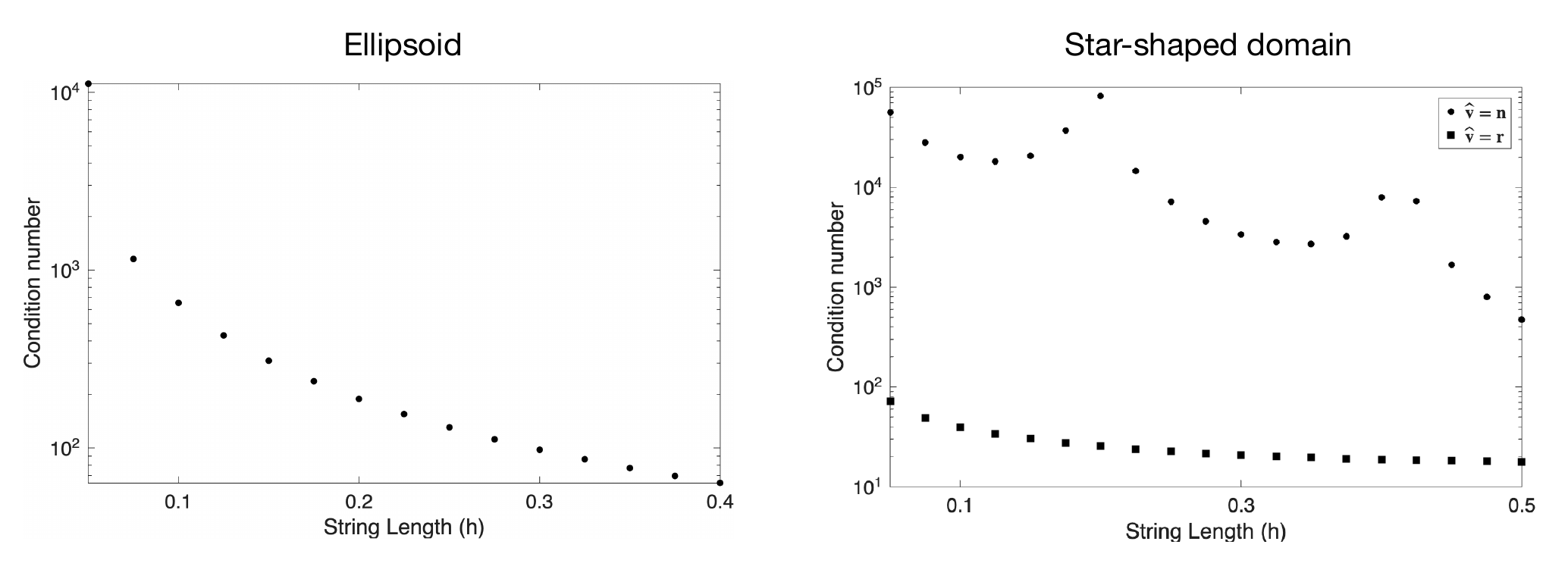}
\caption{(Left) Condition number as a function of $h$ for an ellipsoid with principal semi-axes $(2,1,2)$ in three dimensions, and (right) for the star-shaped domain $\Omegastar$ with $\bv = \n$, and $\bv = \hat{\mathbf{r}}$.}
\label{fig:cond_h}
\end{center}
\end{figure}

Finally, to further illustrate the accuracy of the radial strings on star-shaped domains, we consider a ``random'' star-shaped domain whose parametrization is given by
$$
\gammarandom(t) = \left(1 + 1.5 \sum_{i=1}^{m} \big(a_{i} \cos{(it)} + b_{i} \sin{(it)}\big) \right)\begin{bmatrix}\cos{(t)} \\ \sin{(t)}\end{bmatrix} \,, \quad t\in[0,2\pi) \,,
$$
where $a_{i},b_{i}$ are i.i.d. standard Gaussians, and $m = 25$. Here $\bv = \hat{\mathbf{r}}$ with $h = 1.5$, $\lambda = 10$, and $\mu = 1$. Following an analogous procedure to that used in the prior examples, we generate an artificial solution using $200$ random sources in the exterior and verify that the solution agrees with the exact solution up to a rigid body motion. The solution $u(\x)$ and the pointwise residual $r(\x)$ are plotted in Figure~\ref{fig:bent}.

\begin{figure}[!htbp]
\begin{center}
\includegraphics[width=\linewidth]{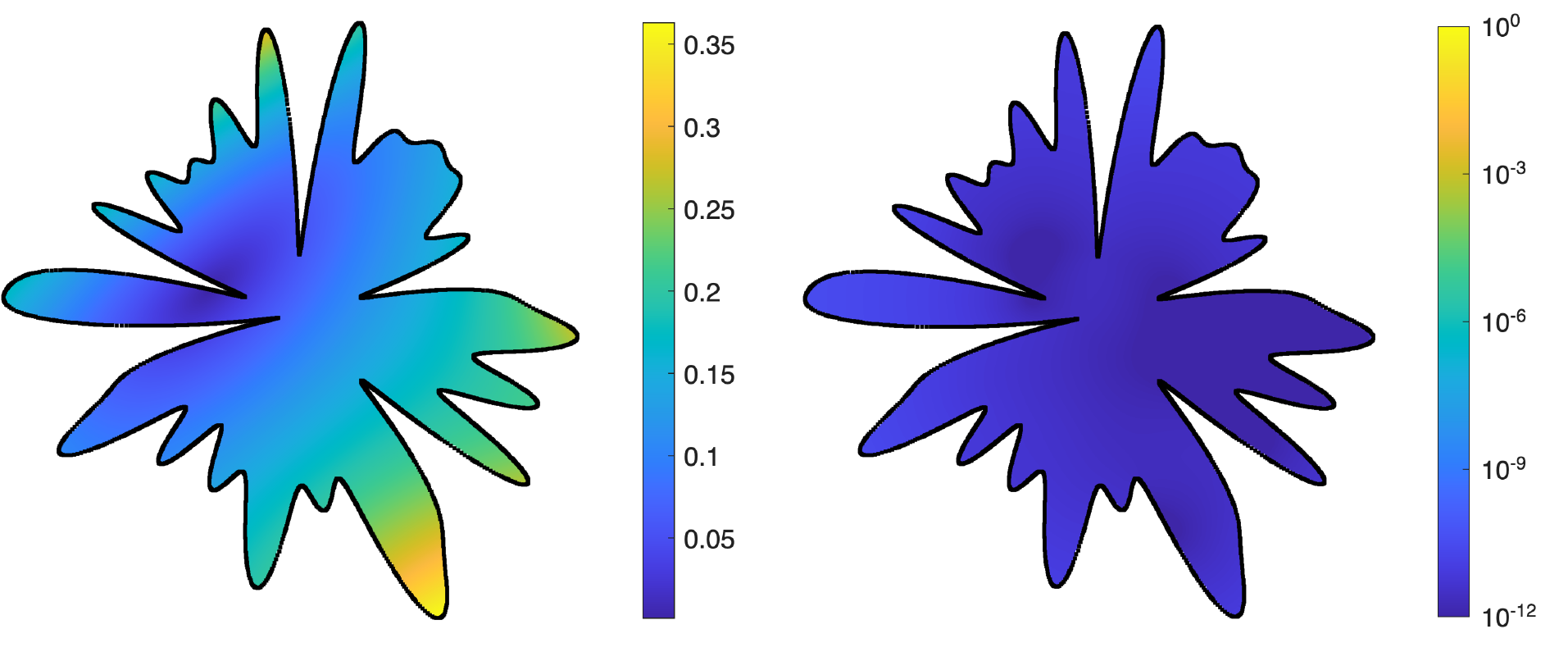}
\caption{(left) computed solution $u(\x)$ on $\Omegarandom$, and (right) relative residual $r(\x)$.}
\label{fig:bent}
\end{center}
\end{figure}

\section{Conclusion}\label{sec:conc}
In this work, we developed new string-based representations for the solution of traction boundary value problems in elastostatics. These representations adapt the Boussinesq-Cerruti solutions which are typically used for the construction of Green's functions in half-spaces to generic exterior and interior boundary value problems. In particular, the singularity in the Boussinesq-Cerruti solutions extends along half-lines in the complementary domain, which we truncate, thereby enabling their use in non-convex domains in the interior and exterior problems. 

The resulting integral equations on the boundary when using these ansatzes are second-kind Fredholm equations. These result in discretized linear systems whose condition number is independent of mesh size. This is illustrated in three dimensions by the constancy of the number of GMRES iterations as the mesh is refined. An additional advantage of integral equation methods is the absence of volumetric locking in the Stokes limit as $\lambda \to \infty$. The integral equation in two dimensions is independent of $\lambda$, and in three dimensions, the condition number varies by less than a factor of $2$ as $\lambda$ is varied over thirteen orders of magnitude.

This string kernel approach can be immediately generalized to several other important applications, including time-harmonic elastodynamics, biharmonic and flexural wave problems~\cite{NEKRASOV2025114355,askham2025surfacelayerslinearizedwater}, acoustic boundary layers~\cite{linden,berggren_acoustic_2018}, and generalized impedance boundary conditions~\cite{Cakoni_2013}. In the above settings, existing integral representations tend to involve singular integral operators (which in many cases are preconditioned using the Hilbert transform to obtain a second-kind integral equation), or involve kernels requiring numerical inversion of Fourier transforms. String kernel representations provide an alternate approach for obtaining second-kind Fredholm equations for these problems, without operator composition.  For scattering problems, string kernel type ansatzes can also be combined with coordinate complexification methods~\cite{epstein2025coordinatecomplexificationhelmholtzequation,HOSKINS2025101721,goodwill2025fastmultipolemethodcomplex} to solve problems with infinite obstacles. All of these directions are being pursued and will be reported at a later date.

\section*{Acknowledgements}
JGH was supported in part by a Sloan Research Fellowship. This research was supported in part by grants from the NSF (DMS-2235451) and Simons Foundation (MPS-NITMB-00005320) to the NSF-Simons National Institute for Theory and Mathematics in Biology (NITMB). AEL acknowledges support from the NSF under award DMS-2052636. MR was supported in part by the Anusandhan National Research Foundation (ANRF/ARG/2025/011989/MS and ANRF/ARGM/2025/002421/MTR). {We would also like to thank the referees whose comments have helped improve the exposition in the manuscript.}

\bibliographystyle{siam}
\bibliography{refs}

\end{document}